\documentclass[]{amsart}

\usepackage{amssymb}
\usepackage{enumerate}

\usepackage[all]{xy}

\newtheorem{proposition}{Proposition}
\newtheorem{theorem}[proposition]{Theorem}

\newtheorem{corollary}[proposition]{Corollary}

\theoremstyle{definition}  
\newtheorem{definition}[proposition]{Definition}
\newtheorem{example}[proposition]{Example}
\newtheorem{remark}[proposition]{Remark}

\newcommand{\thmref}[1]{Theorem~\ref{#1}}
\newcommand{\propref}[1]{Proposition~\ref{#1}}

\newcommand{\cororef}[1]{Corollary~\ref{#1}}

\newcommand{\secat}{{\rm secat}\,}
\newcommand{\cat}{{\rm cat}\,}
\newcommand{\hicat}[1]{{\rm cat}_{#1}\,}

\newcommand{\relcat}{{\rm relcat}\,}
\newcommand{\relcatk}{{\rm relcat}_k\,}
\newcommand{\hirelcat}[1]{{\rm relcat}_{#1}\,}

\newcommand{\id}{{\rm id}}

\renewcommand{\leq}{\leqslant}
\renewcommand{\geq}{\geqslant}
\newcommand{\ent}[1]{\lfloor #1 \rfloor}

\title{About sectional category of the Ganea maps}

\author{Jean-Paul Doeraene}



\begin{document}

\begin{abstract}
We first compute the James' sectional category (secat) of the Ganea map $g_k$ of any map $\iota_X$ in terms of the sectional category of $\iota_X$: We show that $\secat g_k$ is the integer part of $\secat \iota_X / (k+1)$. Next we compute the relative category (relcat) of $g_k$. In order to do this, we introduce the relative category of order $k$ ($\mbox{relcat}_k$) of a map and show that $\relcat g_k$ is the integer part of $\relcatk \iota_X/(k+1)$. Then we establish some inequalities linking secat and relcat of any order: We show that 
$\secat \iota_X \leq \relcatk \iota_X \leq \secat \iota_X + k + 1$
and
$\relcatk \iota_X \leq \hirelcat{k+1} \iota_X \leq \relcatk \iota_X + 1$.
We give examples that show that these inequalities may be strict.
\end{abstract}

\maketitle


In order to compute the `Lusternik-Schnirelmann category' $\cat X$  of a space $X$, Ganea \cite{Gan67} associates  a fibre-cofibre construction  to $X$, more precisely a sequence of fibrations $p_n(X) \colon E_n \to X$ for $n \geq 0$. This invariant for spaces is in some sense extended to maps by the notion of `sectional category' ($\secat$ for short) of a fibration $f$, originally defined by Swarz \cite{Sva66}. There is also a Ganea-type sequence of fibrations $p_n(f)$ associated to $f$ to compute its sectional category. Actually, the  LS-category of $X$ is the sectional category of the path fibration $PX \to X$, so the LS-category is a particular case of sectional category. One can also  define the sectional category of any map as the sectional category of any equivalent fibration; and, in the same way, the sequence of fibrations $p_n$ above can be replaced by a sequence of maps $g_n$, defined up to homotopy. As a particular case, the sectional category of the diagonal map $\Delta\colon X \to X\times X$ is the topological complexity of $X$ defined by Farber \cite{Far03}.

In this paper, we first show that the sectional category of the $n^{\mbox{th}}$ Ganea map $g_n(X)$ of $X$  is the integer part of $\cat X /(n+1)$. More generally, the sectional category of the  Ganea map $g_n(\iota_X)$  associated to any map $\iota_X$ is the integer part of $\secat \iota_X/(n+1)$.

As we may `think of' the sectional category as the degree of obstruction for a map to have a homotopy section, this shows us how this degree of obstruction decreases when we consider the successive Ganea maps. For instance, for a space $X$ with $\cat X = 7$, the successive values of $\secat(g_n(X))$ for $0 \leq n \leq 7$ are
$$7 \quad 3 \quad 2 \quad 1 \quad 1  \quad 1 \quad 1 \quad 0.$$

In \cite{DoeHa13}, we used the same Ganea-type construction to define the `relative category' of a map ($\relcat$ for short). It turns out that the relative category can differ from the sectional category by at most one. More precisely, we have $$\secat \iota_X \leq \relcat \iota_X \leq \secat \iota_X +1.$$ This establishes a dichotomy between maps: those for which the sectional category equals the relative category, and those for which they differ by 1. As a particular case, the relative category of the diagonal map $\Delta\colon X \to X\times X$ is the monoidal topological complexity of $X$ defined in \cite{IwaSak10}.

\smallskip
In this paper we introduce the `relative category of order $k$' ($\mbox{relcat}_k$), and show that the relative category of the  $n^{\mbox{th}}$ Ganea map $g_n(\iota_X)$ associated to a map $\iota_X$ is the integer part of $\hirelcat{n} \iota_X/(n+1)$. 

When $\iota_X \colon \ast \to X$, we write $\relcatk \iota_X = \hicat{k} X$. 

We link all these invariants together by several inequalities: $$\secat \iota_X \leq \relcatk \iota_X \leq \secat \iota_X + k + 1$$ and
$$\relcatk \iota_X \leq \hirelcat{k+1} \iota_X \leq \relcatk \iota_X + 1.$$

Finally, we show that, with some hypothesis on the connexity of $\iota_X$ and the homotopical dimension of the source of $g_n(\iota_X)$,  $\hirelcat{k} \iota_X = \secat \iota_X$ for all $k \leq n$.

For a given space $X$ (respectively: map $\iota_X$), the set of integers $k$ for which the equality
$\hicat{k+1} X = \hicat{k} X$  (respectively: $\hirelcat{k+1} \iota_X = \relcatk \iota_X$) holds is an interesting data of this space (respectively: map).
There are at most  as many such integers as $\cat X$ (respectively: $\relcat \iota_X$).
For instance for $X = K(\mathbb{Q}, 1)$, there is just one such $k$, which is 0, namely:
$$\hicat{0} X = \hicat{1} X = 2\quad\mbox{and}\quad \hicat{k} X = k+1\mbox{ for }k > 1.$$


\section{Sectional category of the Ganea maps}

We use the symbol $\simeq$  both to mean that maps are homotopic, or that spaces are of the same homotopy type.
We denote the integer part of a rational number $q$ by $\ent{q}$.

We build all our spaces and maps with `homotopy commutative diagrams', especially `homotopy pullbacks' and `homotopy pushouts', in the spirit of \cite{Str11}.

\smallskip
Recall the following construction:
\begin{definition}\label{ganea}
For any map $\iota_X\colon A \to X$,
the \emph{Ganea construction} of $\iota_X$
is the following sequence of homotopy commutative diagrams ($i \geq 0$):
$$\xymatrix{
&A\ar[dr]_{\alpha_{i+1}}\ar@/^/[rrrd]^{\iota_X}\\
F_{i}\ar[rd]_{\beta_i}\ar[ur]^{\eta_{i}}&&G_{i+1}\ar@{-->}[rr]|-(.35){g_{i+1}}&&X\\
&G_{i}\ar[ru]^{\gamma_{i}}\ar@/_/[rrru]_{g_{i}}
}$$
where the outside square is a homotopy pullback, the inside square is a homotopy pushout
and the map $g_{i+1} = (g_i,\iota_X) \colon G_{i+1} \to X$
is the whisker map induced by this homotopy pushout.
The iteration starts with  $g_0 = \iota_X \colon A \to X$.
\end{definition}

In other words, the map $g_{i+1}$ is the join of $g_i$ and $\iota_X$ over X, namely $g_{i+1} \simeq g_i \bowtie_X \iota_X$.
When we need to be precise, we denote $G_i$ by $G_i(\iota_X)$ and $g_i$ by $g_i(\iota_X)$. If $A \simeq *$, we also write $G_i(X)$ and $g_i(X)$ respectively.

For coherence, let $\alpha_0 = \id_A$. For any $i \geq 0$, there is a whisker map $\theta_i = (\id_A, \alpha_i)\colon A \to F_i$ induced by the homotopy pullback. Thus $\theta_i$ is a homotopy section of  $\eta_i$. Moreover we have $\gamma_i \circ \alpha_i \simeq \alpha_{i+1}$.

\begin{proposition}For any map $\iota_X\colon A\to X$, we have
$$g_j(g_i(\iota_X)) \simeq g_{ij+i+j}(\iota_X).$$
\end{proposition}

\begin{proof}This is just an application of the `associativity of the join' (see \cite{Doe98}, Theorem 4.8 for instance):
\begin{align*}
g_j(g_i(\iota_X)) &\simeq g_i(\iota_X) \bowtie_X \dots \bowtie_X g_i(\iota_X) \quad\mbox {($j+1$ times)}\\
        &\simeq (\iota_X \bowtie_X \dots \bowtie_X \iota_X) \dots  (\iota_X \bowtie_X \dots \bowtie_X \iota_X)\\
        &\simeq \iota_X \bowtie_X \dots \bowtie_X \iota_X \quad\mbox {($(j+1)(i+1)$ times)}\\
        &\simeq g_{(j+1)(i+1)-1}(\iota_X)
\end{align*}
\end{proof}

\begin{definition}\label{LSganea}
Let $\iota_X\colon A \to X$ be any map. 

1) The \emph{sectional category} of $\iota_X$ is the least integer $n$ such that the map $g_n\colon G_n(\iota_X)\to X$ has a homotopy section, i.e. there exists a map $\sigma\colon X \to G_n(\iota_X)$ such that $g_n \circ \sigma \simeq \id_X$.

2) The \emph{relative category} of $\iota_X$ is the least integer $n$ such that the map $g_n\colon G_n(\iota_X)\to X$ has a homotopy section $\sigma$ and $\sigma \circ \iota_X \simeq \alpha_n$.
\end{definition}

We denote the sectional category  by $\secat(\iota_X)$, and the relative category by $\relcat(\iota_X)$.
If $A = \ast$,  $\secat(\iota_X) = \relcat(\iota_X)$ and is denoted simply by $\cat(X)$; this is the `normalized' version of the Lusternik-Schnirelmann category. 

A lot about the integers cat and secat is collected in \cite{CLOT03}. The integer relcat is introduced in \cite{DoeHa13}, and further studied in \cite{DoeHa13-2} and \cite{CarGarVan14}.  
\smallskip

\begin{proposition}\label{secatgk}For any map $\iota_X\colon A \to X$, we have:
$$\secat g_k(\iota_X) = \big\lfloor{\frac{\secat \iota_X}{k+1}}\big\rfloor$$
\end{proposition}

\begin{proof}By definition, $\secat g_k(\iota_X)$ is the least integer $n$ such that $g_n(g_k(\iota_X))$, i.e. $g_{kn+k+n}(\iota_X)$, has a homotopy section.  Thus, if $\secat \iota_X = m$, $n$ will be such $kn+k+n \geq m$ and $k(n-1)+k+(n-1) < m$, that is $n \geq \frac{m}{k+1} - \frac{k}{k+1}$ and $n < \frac{m}{k+1} +\frac{1}{k+1}$, so $n = \ent{\frac{m}{k+1}}$.
\end{proof}


\section{Higher relative category}

For any map $\iota_X\colon A \to X$ and two integers $0 \leq k < i$, consider the following homotopy commutative diagram
$$\xymatrix{
&G_k\ar[dr]_{\gamma_{k,i}}\ar@/^/[rrrd]^{g_k}\\
H_{i-k-1}^k\ar[rd]\ar[ur]&&G_{i}\ar@{-->}[rr]|-(.35){g_{i}}&&X\\
&G_{i-k-1}\ar[ru]\ar@/_/[rrru]_{g_{i-k-1}}
}$$
where the outside square is a homotopy pullback, the inside square is a homotopy pushout. 

Because of the associativity of the join, we also have $\gamma_{k,i} \simeq \gamma_{i-1} \circ \gamma_{i-2} \circ \dots \circ \gamma_{k+1} \circ \gamma_k$. For coherence, let $\gamma_{k,k} = \id_{G_k}$. 

\begin{definition}
Let $\iota_X\colon A \to X$ be any map. 
The \emph{relative category of order $k$} of $\iota_X$ is the least integer $n \geq k$ such that the map $g_n\colon G_n(\iota_X)\to X$ has a homotopy section $\sigma$ and $\sigma \circ  g_k \simeq \gamma_{k,n}$.
\end{definition}

We denote this integer by $\relcatk \iota_X$. According to the convention to avoid the prefix `rel' when $A \simeq *$, we write $\hicat{k} X = \relcatk \iota_X$ in this case.

\begin{remark}
Notice that $\hirelcat{0} \iota_X = \relcat \iota_X$ and that, clearly, $k \leq \hirelcat{k} \iota_X   \leq \hirelcat{k+1} \iota_X$ for any $k$.  Also notice that $\relcatk \iota_X = k$ if and only if $g_k(\iota_X)$ is a homotopy equivalence. In particular, $\hicat{k} \ast = k$ for any $k$.
\end{remark}

Following the same reasoning as in Proposition \ref{secatgk}, we have:
\begin{proposition}\label{relcatgk}For any map $\iota_X\colon A \to X$, we have:
$$\relcat g_k(\iota_X) = \big\lfloor{\frac{\relcatk \iota_X}{k+1}}\big\rfloor$$
\end{proposition}

\begin{proposition}\label{secatplus}For any map $\iota_X\colon A \to X$, any $k$, we have:
$$\secat \iota_X \leq \relcatk \iota_X \leq \secat \iota_X + k + 1.$$
\end{proposition}

\begin{proof}Only the second inequality needs a proof. Let $n = \secat \iota_X$ et let $\sigma$ be a homotopy section of $g_n$. Consider the following homotopy commutative diagram:
$$\xymatrix{
G_k\ar[dd]_{g_k}\ar[rr]^{\sigma'}&& H_n^k\ar[dd]\ar[rr]^{g'} && G_k\ar[dd]^{g_k}\ar[ld]_{\gamma_{k,n+k+1}}\\
&&&G_{n+k+1}\ar@{-->}[rd]^{g_{n+k+1}}\\
X\ar[rr]_\sigma&&G_n\ar[rr]_{g_n}\ar[ru]_{\gamma_{n,n+k+1}}&&X
}$$
where the two squares are homotopy pullbacks. We have $g' \circ \sigma' \simeq \id_{G_k}$ by the Prism lemma (see \cite{Doe98}, Lemma 1.3 for instance). The map $\sigma^+ = \gamma_{n,n+k+1} \circ \sigma$ is a homotopy section of $g_{n+k+1}$ and, moreover, $\sigma^+ \circ g_k \simeq \gamma_{k,n+k+1} \circ g' \circ \sigma' \circ \gamma_{k,n+k+1}$. So $\relcatk \iota_X \leq n+k+1$.
\end{proof}

\begin{theorem}\label{relcatkplus1}
For any map $\iota_X\colon A \to X$, any $k$, we have:
$$\relcatk \iota_X \leq \hirelcat{k+1} \iota_X \leq \relcatk \iota_X + 1.$$
\end{theorem}

\begin{proof}Only the second inequality needs a proof.  Let $n= \relcatk \iota_X$ et let $\sigma$ be a homotopy section of $g_n$ such that $\sigma \circ g_k \simeq \gamma_{k,n}$. Consider the following homotopy commutative diagram:
$$\xymatrix{
&&&F_k\ar[llldd]\ar[ddd]|(.64)\hole\ar[rdd]\ar[rrrrdd]\\
\\
A\ar[rrrr]^{\bar\sigma}\ar[rd]\ar[ddd]&&&&F_n\ar[ddd]\ar[rrr]^{\eta_n}&&&A\ar[lld]\ar[ldd]\ar[ddd]\\
&G_{k+1}\ar@{-->}[ldd]\ar@/_2pc/@{=}[rrrr] &&G_k\ar[ll]\ar[rdd]|(.3)\hole\ar[llldd]|(.25)\hole\ar[rrrrdd]|(.24)\hole|(.62)\hole\ar[rr]|(.47)\hole &&G_{k+1}\ar@{-->}[rd]\\
&&&&&&G_{n+1}\ar@{-->}[rd]\\
X\ar[rrrr]_\sigma&&&&G_n\ar[rru]_{\gamma_n}\ar[rrr]_{g_n}&&&X
}$$
The map $\sigma^+ = \gamma_n \circ \sigma$ is a homotopy section of $g_{n+1}$ and $\sigma^+ \circ g_{k+1} \simeq \gamma_{k+1,n+1}$, so $\hirelcat{k+1} \iota_X \leq n+1$. 
\end{proof}

\begin{corollary}\label{relcatplusk}
For any map $\iota_X\colon A \to X$, any $k$, we have:
$$\relcat \iota_X \leq \relcatk \iota_X \leq \relcat \iota_X + k.$$
\end{corollary}

\begin{remark}As a consequence of \thmref{relcatkplus1} and  \cororef{relcatplusk}, if $n=\relcat \iota_X$, there are at most $n$ integers $k$ for which $\hirelcat{k+1} \iota_X = \relcatk \iota_X$.
\end{remark}
  
\begin{example}If $\iota_X$ is a homotopy equivalence, then $g_k$ is a homotopy equivalence for all $k$.  So $\relcatk \iota_X = k$ for all $k$.
\end{example}

\begin{example}Let $A \not\simeq *$  and consider the map $\iota_* \colon A \to *$. We have $\secat \iota_* = 0$ because $\iota_*$ has a (unique) section.
By Proposition \ref{secatplus}, $\relcatk \iota_* = k$ or $1+k$. 
Indeed, for any k, the map $\gamma_{k,k+1} \colon A \bowtie \dots \bowtie A \mbox{(k+1 times)} \to A \bowtie \dots \bowtie A \mbox{(k+2 times)}$ is homotopic to the null map, so $\sigma \circ g_k \simeq \gamma_{k,k+1}$ where $\sigma\colon * \to G_{k+1}(\iota_*)$.
But we cannot have $\relcatk \iota_* = k$ unless $g_k(\iota_*)\colon A \bowtie \dots \bowtie A  \mbox{(k+1 times)} \to *$ is  a homotopy equivalence.

For instance if $A$ is the Epstein's space (such that $A\not\simeq *$  but $\Sigma A \simeq *$), then $A \bowtie A \simeq \Sigma A \wedge A \simeq *$ and $g_k$ is a homotopy equivalence for all $k > 0$, so $\hirelcat{0} \iota_* = 1$ and $\relcatk \iota_* = k$ par all $k > 0$.
However if we chose a  simply-connected CW-complex $A$ (in order that $A \bowtie  \dots \bowtie A \not\simeq *$), then $\relcatk \iota_* = k+1$ for all $k$.

\end{example}

\begin{example}Consider any CW-complex $X$ with $\cat X = 1$  and the map $\iota_{X}\colon * \to X$. We have $\secat \iota_{X} = \relcat \iota_{X} = \cat {X} = 1$. Let us compute $\hicat{1} X = \hirelcat{1} \iota_X$. Notice that $G_1(X) \simeq \Sigma\Omega X$. By \thmref{relcatkplus1}, we know that $1 \leq \hicat{1} X \leq 2$. By the way, we can say that $\gamma_{1,2}\colon \Sigma\Omega X \to G_2(X)$ factorizes up to homotopy through $g_1\colon \Sigma\Omega X \to X$. But we cannot have  $\hicat{1} X = 1$ because $g_1$ is not a homotopy equivalence; so $\hicat{1} X  = 2$.
\end{example}

\begin{example}More generally, if $\relcat \iota_X = 1$, we have $k \leq \relcatk \iota_X \leq k+1$ for any $k$. So $\relcatk \iota_X =k+1$ unless $g_k(\iota_X)$ is a homotopy equivalence.
\end{example}

Let be given any map $\iota_X\colon A \to X$ with $\secat(\iota_X) \leq n$ and any homotopy section $\sigma\colon X \to G_n$ of $g_n\colon G_n\to X$.  Consider the following homotopy pullbacks:
$$\xymatrix{
Q\ar[d]_{\pi'}\ar[r]^\pi&G_k\ar[d]_{\theta_n^k}\ar@{=}[rd]\\
G_k\ar[r]_{\bar \sigma}\ar[d]_{g_k}&H_n^k\ar[r]_{\eta_n^k}\ar[d]_{h_n^k}&G_k\ar[d]^{g_k}\\
X\ar[r]_\sigma&G_n\ar[r]_{g_n}&X
}$$
where  $\theta_n^k = (\gamma_{k,n}, \id_{G_k})$ is the whisker map induced by the homotopy pullback $H_n^k$. By the Prism lemma, we know that the homotopy pullback of $\sigma$ and $h_n^k$ is indeed $G_k$, and that $\eta_n^k \circ \bar\sigma \simeq \id_{G_k}$.
Also notice that $\pi \simeq \pi'$ since $\pi \simeq \eta_n^k \circ \theta_n^k \circ \pi \simeq \eta_n^k \circ \bar\sigma \circ \pi' \simeq \pi'$.

\begin{proposition}\label{sigmabar}Let be given any map $\iota_X\colon A \to X$ with $\secat(\iota_X) \leq n$ and any homotopy section $\sigma\colon X \to G_n$ of $g_n\colon G_n\to X$.  With the same definitions and notations as above, the following conditions are equivalent:
\begin{enumerate}
\item $\sigma \circ g_k \simeq \gamma_{k,n}$.
\item $\pi$ has a homotopy section.
\item $\pi$ is a homotopy epimorphism.
\item $\theta_n^k \simeq \bar\sigma$.
\end{enumerate}
\end{proposition}

\begin{proof}We have the following sequence of implications:

(i) $\implies$ (ii):  Since $\sigma \circ g_k \simeq \gamma_{k,n}  \simeq h_n^k \circ \theta_n^k \circ \id_{G_k}$, we have a whisker map $(g_k,\id_{G_k})\colon {G_k}\to Q$ induced by the homotopy pullback $Q$ which is a homotopy section of $\pi$.

(ii) $\implies$ (iii): Obvious.

(iii) $\implies$ (iv): We have $\theta_n^k \circ \pi \simeq  \bar\sigma \circ \pi' \simeq  \bar\sigma \circ \pi$ since $\pi \simeq \pi'$. Thus  $\theta_n^k \simeq \bar\sigma$ since $\pi$ is a homotopy epimorphism.

(iv) $\implies$ (i): We have $\sigma \circ g_k \simeq h_n^k \circ \bar\sigma \simeq h_n^k \circ \theta_n^k \simeq \gamma_{k,n}$.
\end{proof}

\begin{theorem}\label{dimension}Let be   a $(q-1)$-connected map $\iota_X\colon A\to X$ with $\secat \iota_X = n$. If $G_k$ has the homotopy type of a CW-complex with $\dim G_k < (n + 1)q-1$ then $\sigma \circ g_k \simeq \gamma_{k,n}$ for any homotopy section $\sigma$ of $g_n$, so $\hirelcat{i} \iota_X = \secat \iota_X$ for all $i \leq k$.\end{theorem}

\begin{proof}Recall that $g_i$ is the $(i+1)$-fold join of $\iota_X$. Thus by \cite{Mat76}, Theorem 47, we obtain that, for each $i \geq 0$, $g_i : G_i \to X$ is $(i+1)q-1$-connected. As $g_i$ and $\eta_i^k$ have the same homotopy fibre, the Five lemma implies that $\eta_i^k\colon H_i^k \to G_k$ is $(i+1)q-1$-connected, too. By \cite{Whi78}, Theorem IV.7.16, this means that for every CW-complex $K$ with $\dim K < (i+1)q-1$, $\eta_i^k$ induces a one-to-one correspondence $[K,H_i^k] \to [K,G_k]$. Apply this to $K\simeq G_k$ and $i=n$: Since $\theta_n^k$ and $\bar\sigma$ are both homotopy sections of $\eta_n^k$, we obtain $\theta_n^k \simeq \bar\sigma$, and \propref{sigmabar} gives the desired result.
\end{proof}

\begin{example}Let $X$ be the Eilenberg-Mac~Lane space $K(\mathbb{Q}, 1)$. It is known that $\cat(X) = 2$ and that $G_1(X) \simeq \Sigma\Omega X$ has the homotopy type of a wedge of circles (see \cite{CLOT03}, Example 1.9 and Remark 1.62 for instance). By \thmref{relcatkplus1}, we know that $2 \leq \hicat{1} X \leq 3$.  Because $\dim G_1(X) = 1 <  (\cat X+1)-1 = 2$, we have $\sigma \circ g_1 \simeq \gamma_{1,2}$ for any homotopy section $\sigma$ of $g_2(X)$ and $\hicat{1} X = 2$. Moreover $g_k$ is never a homotopy equivalence, so $\hicat{k} X > k$ for any $k$, so $\hicat{k} X = k+1$ for $k \geq 1$.
\end{example}

\bibliographystyle{plain}
\bibliography{secat}

\end{document}